\documentclass{amsart}
\usepackage[T1]{fontenc}
\usepackage{amssymb, amsmath, color, textcomp, url,tikz,
mathtools}
\usepackage[all]{xy}
\usepackage{tikz-cd}
\usepackage{hyperref}
\usepackage[retainorgcmds]{IEEEtrantools}

\usetikzlibrary{matrix,arrows}
\usetikzlibrary{fit}
\usetikzlibrary{positioning}
\usepackage[labelformat=empty]{caption}

\usepackage[normalem]{ulem}
\usepackage{soul}
\usepackage{mdwlist}

\usepackage{enumerate,xspace}
\usepackage{chngcntr, csquotes}
\theoremstyle{plain}
\newtheorem{theorem}{Theorem}[section]
\newtheorem{cor}[theorem]{Corollary}
\newtheorem{prop}[theorem]{Proposition}
\newtheorem{lemma}[theorem]{Lemma}

\newcounter{proofcount}
\AtBeginEnvironment{proof}{\stepcounter{proofcount}}
\newtheorem{claim}{Claim}

\newtheorem*{claim*}{Claim}
\makeatletter                  
\@addtoreset{claim}{proofcount}
\makeatother                   

\newenvironment{claimproof}[1][Proof of Claim \theclaim.]
 {%
	\proof[#1]%

}
{%
	\endproof%
}

\newenvironment{claimproofstar}[1][Proof of the Claim.]
{%
	\proof[#1]%
	
}
{%
	\endproof%
}

\theoremstyle{definition}
\newtheorem{remark}[theorem]{Remark}
\newtheorem{fact}[theorem]{Fact}
\newtheorem{definition}[theorem]{Definition}
\newtheorem{example}[theorem]{Example}

\newtheorem*{notation}{Notation}

\newcommand{\nc}{\newcommand}

\nc{\Z}{\mathbb{Z}}
\nc{\Q}{\mathbb{Q}}
\nc{\N}{\mathbb{N}}
\nc{\F}{\mathbb{F}}
\nc{\UU}{\mathbb{U}}
\nc{\C}{\mathbb{C}}
\nc{\CC}{\mathcal C}

\nc{\M}{\mathcal{M}}
\nc\LL{\mathcal L}

\nc{\sbgp}[1]{\langle\xspace {#1}\xspace\rangle}

\nc{\dcl}{\operatorname{dcl}}
\nc{\acl}{\operatorname{acl}}
\nc{\sep}{\operatorname{sep}}
\nc{\alg}{\operatorname{alg}}
\nc{\ins}{\operatorname{ins}}
\nc{\ACF}{\operatorname{ACF}}
\nc{\SCF}{\operatorname{SCF}}
\nc{\DCF}{\operatorname{DCF}}
\nc{\ld}{\operatorname{ld}}

\nc{\nf}[1]{_{\mid {#1}}}
\nc{\restr}[1]{\xspace_{\upharpoonright {#1}}}
\nc{\diff}[2]{#1\langle #2 \rangle}

\nc\inv{ ^{-1}}
\nc\Coh{\operatorname{Coh}}

\nc{\supp}{\operatorname{\mathrm{supp}}}

\nc{\tp}{\operatorname{tp}}
\nc{\trdeg}{\operatorname{trdeg}}

\nc\cb{\operatorname{Cb}}
\nc\U{\operatorname{U}}
\nc{\cf}{\text{cf.\,}}
\nc{\eg}{\text{e.g. }}

\reversemarginpar

\makeatletter
\def\moverlay{\mathpalette\mov@rlay}
\def\mov@rlay#1#2{\leavevmode\vtop{%
   \baselineskip\z@skip \lineskiplimit-\maxdimen
   \ialign{\hfil$\m@th#1##$\hfil\cr#2\crcr}}}
\newcommand{\charfusion}[3][\mathord]{
    #1{\ifx#1\mathop\vphantom{#2}\fi
        \mathpalette\mov@rlay{#2\cr#3}
      }
    \ifx#1\mathop\expandafter\displaylimits\fi}
\makeatother

\newcommand{\cupdot}{\charfusion[\mathbin]{\cup}{}}

\def\Ind#1#2{#1\setbox0=\hbox{$#1x$}\kern\wd0\hbox to
  0pt{\hss$#1\mid$\hss} \lower.9\ht0\hbox to
  0pt{\hss$#1\smile$\hss}\kern\wd0}
\def\Notind#1#2{#1\setbox0=\hbox{$#1x$}\kern\wd0\hbox to
  0pt{\mathchardef\nn="0236\hss$#1\nn$\kern1.4\wd0\hss}\hbox to
  0pt{\hss$#1\mid$\hss}\lower.9\ht0 \hbox to
  0pt{\hss$#1\smile$\hss}\kern\wd0}
\def\ind{\mathop{\mathpalette\Ind{}}}
\def\nind{\mathop{\mathpalette\Notind{}}}

\nc{\indacf}{\ind^{\mathrm{ACF}}}
\nc{\indscf}{\ind^{\mathrm{SCF}}}
\nc{\inddcf}{\ind^{\mathrm{DCF}}}
\nc{\indp}{\ind^{p}}
\nc{\indld}{\ind^{\ld}}

\def\nindld{\mathop{\ \ \hbox to 0pt{\hss$\!\not{\mid}^{\hbox to
0pt{$\scriptstyle\, \mathrm{ld}$\hss}}$\hss}
\lower4pt\hbox to 0pt{\hss$\smile$\hss}\ \ }}

\def\nindacf{\mathop{\ \ \hbox to 0pt{\hss$\!\not{\mid}^{\hbox to
0pt{$\scriptstyle\, \mathrm{ACF}$\hss}}$\hss}
\lower4pt\hbox to 0pt{\hss$\smile$\hss}\ \hskip3mm \ }}

\def\nindscf{\mathop{\ \ \hbox to 0pt{\hss$\!\not{\mid}^{\hbox to
0pt{$\scriptstyle\, \mathrm{sep}$\hss}}$\hss}
\lower4pt\hbox to 0pt{\hss$\smile$\hss}\ \hskip3mm \ }}

\def\nindp{\mathop{\ \ \hbox to 0pt{\hss$\!\not{\mid}^{\hbox to
0pt{$\scriptstyle\, p$\hss}}$\hss}
\lower4pt\hbox to 0pt{\hss$\smile$\hss}\ \hspace{3mm} }}

\begin{document}

\title[Forking independence in DCF$_{p.m}$]{Forking independence in differentially closed fields of positive characteristic}
\date{\today}

\author{Piotr Kowalski, Omar Le\'on S\'anchez and Amador Martin-Pizarro}

\address{Instytut Matematyczny, Uniwersytet Wroc\l awski, pl. Grunwaldzki 2--4, 50384 Wroc\l aw, Poland}
\email{pkowa@math.uni.wroc.pl}

\address{Department of Mathematics, University of Manchester, Oxford Road, M13 9PL, Manchester, United Kingdom}
\email{omar.sanchez@manchester.ac.uk}

\address{Abteilung f\"ur Mathematische Logik, Mathematisches Institut,
  Albert-Ludwigs-Universit\"at Freiburg, Ernst-Zermelo-Stra\ss e 1,
  D-79104
  Freiburg, Germany}
\email{pizarro@math.uni-freiburg.de}

\thanks{{\em Acknowledgements}: The first author is supported by the Narodowe Centrum Nauki grant no. 2021/43/B/ST1/00405 and by the T\"{u}bitak grant no. 1001-124F359.
\\
The second author was partially supported by EPSRC grant EP/V03619X/1}

\begin{abstract}
We provide a differential-algebraic description of forking independence in the stable theory DCF$_{p,m}$ of differentially closed fields of characteristic $p>0$ with $m$-many commuting derivations. As a by-product of this description, we prove that types over algebraically closed subsets of the real sort are stationary.
In addition, we prove that the set of non-zero solutions to the Bernoulli differential equation $x'=x^{p^k+1}$ with $k>0$ is strongly minimal and its geometry is strictly disintegrated, which implies that this set is algebraically independent over $\F_p$.
\end{abstract}

\maketitle

\tableofcontents

\section*{Introduction}

It is fair to say that, compared to DCF$_0$, the theory DCF$_p$ of differentially closed fields of characteristic $p>0$ remains vastly unexplored. For instance, while in the 1970's Shelah observed that DCF$_p$ is a stable theory by a counting-type argument \cite{sS73}, no algebraic description of forking independence had been provided (until now). This sits in contrast to the case of  DCF$_0$ (and other similar theories of fields with operators \cite{MS14}) where forking independence was described solely in terms of the associated reduct to the pure field language; in other words, in terms of algebraic independence.

Recall that the underlying field of a model of DCF$_p$ is a separably closed field of infinite degree of imperfection; namely, a model of SCF$_{p,\infty}$. So, one could be tempted to think that forking independence in DCF$_p$ should be described by forking independence of SCF$_{p,\infty}$, at least for (model-theoretically) algebraically closed sets.
However, our results show that forking independence
in DCF$_p$ cannot be described in terms of algebraic independence alone or even in terms of $p$-independence, which is part of forking independence in the reduct SCF$_{p,\infty}$, as shown by Srour \cite{gS86}. We provide explicit examples (see Remark \ref{R:remark_indstar}) illustrating that forking independence in DCF$_p$ does not imply $p$-independence nor is equivalent to algebraic independence.

The motivating question that drives the results of this note is: besides algebraic independence, what other (differential-)algebraic conditions follow from forking independence  $K\inddcf_k L$ on the definably closed subfields $k$, $K$ and $L$? As we mentioned above, forking independence is not as strong as the $p$-independence of $K$ and $L$ over $k$, and it is worth noting that in the differential context $p$-independence of $K$ and $L$ over $k$ translates to the field compositum $KL$ being differentially perfect, that is, its field of constants $\CC_{KL}$ equals $(KL)^p$.

Whilst $K\inddcf_k L$ need not imply that the compositum $KL$ is differentially perfect, it does imply an \emph{almost} differentially perfect condition. Namely,  there exists a $p$-basis of $\CC_{KL}^{1/p}$ over $KL$ which is differentially separably independent over $KL$ (see Section~\ref{S:DCFp} for the corresponding definitions). We give the latter property a name and say that $KL$ is \emph{differentially transcendentally almost perfect}, or \emph{differentially trap} for short. In Lemma~\ref{L:differentially traplem} we show that after adding all derivatives of the $p$-basis, the corresponding differential field is differentially perfect (hence justifying the terminology differentially trap).

\medskip

The main result of the paper is that differentially trap is the missing condition to describe forking independence $\inddcf$ in the theory DCF$_{p,m}$ of differentially closed fields of characteristic $p>0$ equipped with $m$ commuting derivations. In Theorem~\ref{T:main} we prove:

\medskip

\noindent {\bf Theorem A. }{\it Consider a saturated ambient model $\UU$ of the theory $\DCF_{p,m}$.
Given definably closed differential fields $k$, $K$ and $L$ with $k\subseteq K\cap L$, we have that  $K\inddcf_k L$ if and only if \begin{itemize} \item the fields $K$ and $L$ are algebraically independent over $k$ and
\item the differential field $KL$ is differentially trap.
\end{itemize} Furthermore, types over algebraically closed subsets of the real sort are stationary.}
\medskip

Our description of forking independence allows us to investigate the existence of minimal types and strongly minimal subsets in DCF$_p$ in Section~\ref{S:Ber}. In contrast to the case of SCF$_{p,\infty}$, where no strongly minimal definable sets exist, we show that such sets do exist in DCF$_p$. Our examples come from the classical Bernoulli equation $x'=x^n$ and they also enjoy a strong ``Ax-Schanuel like'' property.

\medskip

\noindent {\bf Theorem B.} {\it For any positive integer $k$, the definable set \[\left\{a\in \UU\setminus \{0\} \mid \partial(a)=a^{p^k+1}\right\}\] is strongly minimal and its induced geometry is strictly disintegrated. In particular,  this set is algebraically independent over $\F_p$.}

\medskip

Finally, in the Appendix, we provide a stability criterion suitable for certain theories of fields with operators. Mimicking \cite[Theorem 9]{sS73} and \cite[Theorem 2.4]{pK05}, we use a counting-type argument to deduce stability, and we expect that it may be deployed to prove stability of various theories of fields in positive characteristic equipped with free operators (see \cite{BHKK}).

\medskip

The first author would like to take the opportunity to use the results in this paper as a corrigendum to some previously stated results:
\begin{itemize}
  \item It was wrongly claimed in \cite{pK05} that $\ind^{\DCF}$ can be described as in Fact~\ref{F:Shelah} below over \emph{arbitrary} base sets. Unfortunately, this mistake spread to other papers (as \cite{BHKK} or \cite{goko1}) treating more general cases of operators in the positive characteristic case. Whilst we believe that a description of forking analogous to the one given in Theorem~\ref{T:main} may adapt to these more general cases, we have not yet checked it thoroughly and leave it for future work.

\item The stationarity over algebraically closed subsets of the real sort was already used without a proper argument in \cite{HK4}.
\end{itemize}

\bigskip

\noindent {\bf Acknowledgements.}
The authors would like to thank Jakub Gogolok, Moshe Kamensky and Martin Ziegler for their comments.
We would also like to thank the members of the
model theory group in Wroc{\l}aw for their constructive remarks during the talks of the first author at the model theory
seminars at Wroc{\l}aw University.

The authors would like to thank the support of EPSRC UK, via the second author's grant EP/V03619X/1, that funded a collaborative research visit at the University of Manchester in November 2024 where part of this research was conducted.

\section{Fields and separable extensions}\label{S:Fields}

Though most of the notions and results in this section are classical, we believe that there are some subtleties (especially regarding the role of which subfields of $p^\text{th}$-powers appear) which can be easily overseen. Therefore, we have decided to include some of the definitions.

All throughout this section, we will work inside an ambient field $F$ of characteristic $p>0$.  All tuples and subsets are taken within $F$. Given a subfield $L$ of $F$, we denote by  $L^p$  the subfield of $p^\text{th}$-powers of $L$. Henceforth, we fix two subfields $K$ and $L$ of $F$ with a common subfield $k\subseteq K\cap L$.

\begin{notation}
To render the notation simpler, we will denote the compositum of two fields $K_1$ and $K_2$ by the concatenation $K_1K_2$, in contrast to the use of concatenation in model theory for the set-theoretic union. In order to avoid any possible misunderstanding, we will always use the symbol $\cup$ when writing unions.
\end{notation}

\begin{definition}\label{D:lindisj} \
\begin{enumerate}
\item [(a)]The field $K$ is \emph{linearly disjoint from} $L$ over $k$, denoted by $K\indld_k L$, if whenever a subset $S$ of $K$ is $k$-linearly independent, then it remains so over $L$ (seen as a subset of $F$).

\item [(b)]A set $S$ is \emph{algebraically independent} over $k$ if no element $s$ in $S$ belongs to the algebraic closure $k(S\setminus \{s\})^{\alg}$. We say that $K$ is \emph{algebraically independent} (or \emph{free}, following Lang's terminology) \emph{from} $L$ over $k$, denoted by $K\indacf_k L$, if every  subset $S$ of $K$ which is algebraically independent over $k$ remains so over $L$, or equivalently, if \[ \trdeg(a_1,\ldots, a_n/k)=\trdeg(a_1,\ldots, a_n/L) \text{ for all $a_1,\ldots, a_n$ in $K$}.\]

\item [(c)]A subset $S$ of $K$ is \emph{$p$-independent} over $k$ in $K$ if no element $s$ in $S$ belongs to the subfield $K^p k(S\setminus\{s\})$. If $k=\F_p$, we simply say that $S$ is \emph{$p$-independent} in $K$. A maximal $p$-independent subset $S$ of $K$ over $k$ is called a \emph{$p$-basis} of $K$ over $k$.

\item [(d)]The subset $S$ of $K$ is \emph{separably independent} over $k$
 if no element $s$ in $S$ belongs to the separable closure $k(S\setminus \{s\})^{\sep}$.

\item [(e)]The field extension $k\subseteq K$ is \emph{separable} if $k\indld_{k^p} K^p$, or equivalently, if every $p$-independent subset $S$ of $k$ remains $p$-independent in $K$.
\end{enumerate}
 \end{definition}

\begin{fact}\label{F:sep_indep}~
\begin{enumerate}[(a)]
    \item Linear disjointness and algebraic independence are both  transitive and symmetric notions of independence. Linear disjointness always implies algebraic independence and they coincide whenever one  of the two field extensions $k\subseteq K$ or $k\subseteq L$ is regular \cite[Chapter VIII \S 4]{sL02}.

    \item Algebraic independence implies separable independence. Moreover, the extension $k\subseteq K$ is separable if and only if every subset of $K$ which is separably independent over $k$ is  algebraically independent over $k$ as well \cite[\S 0.2]{eK73}.
\end{enumerate}
\end{fact}
Unlike algebraic independence or linear disjointness, separable independence need not be transitive: Given  a transcendental element $a$ over $\F_p$, the element $a$ is separably independent over $\F_p$ by Fact \ref{F:sep_indep} (b). Similarly, the element $a^{1/p}$ is separably independent over $\F_p(a)$, for it is purely inseparable. However, the subset $A=\{a,a^{1/p}\}$ is clearly separably dependent over $\F_p$. In particular, separable closure is an operator which does not satisfy the \textit{Steinitz exchange} property.

Despite the lack of transitivity, the following result holds.
\begin{lemma}\label{L:sepalg}
If the subset $A$ of $K$ is separably independent over $k$ and the subset $B$ of $K$ is algebraically independent over $k(A)$, then $A\cup B$ is separably independent over $k$.
\end{lemma}
\begin{proof}
Without loss of generality, we need only show that no element $a$ in $A$ can be separable algebraic over $k(A\setminus\{a\}\cup B)$. Assume that $f(\bar a, \bar b, a)=0$ for some tuples $\bar a$ in $A\setminus \{a\}$ and $\bar b$ in $B$ as well as some  non-trivial polynomial over  $k$
\[f(\bar X, \bar Y, Z)=\sum_{i} f_i(\bar X, \bar Y)Z^i.\]
Since the tuple $\bar b$ consists of algebraically independent elements over $k(\bar a, a)$, we obtain $f(\bar a, \bar Y, a)=0$, which translates into a finite system of polynomial equations \[ g_1(\bar a, a)=\cdots=g_m(\bar a, a)=0 \] with coefficients in $k$. Now, the subset $A$ is separably independent over $k$, so  each polynomial $g_i(\bar a, Z)$ in the previous system is either trivial or inseparable in $Z$. In both cases, we have that \[ \frac{\partial f}{\partial Z}(\bar a,\bar Y, Z)=0,\]
and hence the corresponding partial derivative is $0$ when evaluated at $\bar b$, so the polynomial $f(\bar a, \bar b, Z)$ is not separable.  We conclude that $A\cup B$ is separably independent over $k$, as desired.
\end{proof}
\begin{remark}\label{R:sep_ind_p-indep}
Given a field extension $k\subseteq K$ and a subset $A$ of $K$, we have that $A$ is separably independent over $k$ if and only if $A$ is $p$-independent in $k(A)$ over $k$.
\end{remark}
\begin{proof}
Notice first that an element $b$ of $K$ is separable algebraic over $k$ if and only if $b$ belongs to $k(b^p)$ (see for instance \cite[\S1.6]{zC2001}).

Suppose first that $A$ is separably independent over $k$, yet the element $a$ in $A$ belongs to $k(A)^pk(A\setminus\{a\})=k(A\setminus \{a\})(a^p)$. By the previous observation, we deduce that  $a$ belongs to $k(A\setminus \{a\})^{\sep}$, which gives the desired contradiction. For the converse, assume for a contradiction that $A$ is separably dependent over $k$.  Hence, there is some $a$ in $A$ which is separable algebraic over $k(A\setminus \{a\})$. In particular, we have that $a$ belongs to
\[k(A\setminus \{a\})(a^p)= k(A)^pk(A\setminus \{a\})\subseteq K^p k(A\setminus \{a\}),\]
so $A$ cannot be $p$-independent in $K$ over $k$.
\end{proof}

\begin{definition}\label{D:p_indep}
  We say that $K$ is \emph{$p$-independent from $L$} over $k$ in $F$ if every subset of $K$ which $p$-independent over $k$ in $F$ remains $p$-independent over $L$ in $F$.
\end{definition}

Whilst linear disjointness or algebraic independence are absolute notions, notice that $p$-independence depends on the ambient field $F$. Furthermore, we have that $p$-independence corresponds to the independence in the pregeometry given by the $p$-closure in $F$, where $a$ belongs to the $p$-closure of $B$ if $a$ belongs to $F^p(B)$.
\begin{remark}\label{R:pindep_obs}
The above definition of $p$-independence extends Srour's ternary $p$-independence relation \cite[p. 722]{gS86}, for he assumes that all field extensions $k\subseteq K\subseteq  F$, and $k\subseteq L\subseteq  F$ are separable.

Chatzidakis \cite[\S 1.2]{zC99} defines a similar ternary notion of $p$-independence, though replacing the role of $F$ with the compositum $K L$. In her work, the extensions $k\subseteq K\subseteq K L$ and $k\subseteq L$ are assumed to be separable, explicitly or implicitly in the case of $K\subseteq K L$. Under these additional conditions, her definition coincides with our Definition \ref{D:p_indep} if $F=KL$.
\end{remark}

From now on, we work inside a sufficiently saturated model $\UU$ of the $\LL_\mathrm{Ring}$-theory SCF$_{p,\infty}$ of separably closed fields of characteristic $p>0$ and infinite imperfection degree (in other words, from now on $F=\UU$). As shown by Ershov \cite{yE68}, this theory is complete. Wood \cite[Theorem 3]{cW79} showed that the theory  SCF$_{p,\infty}$ is stable, so Shelah's non-forking defines a tame notion of independence, which we will denote by $\indscf \ $.

\begin{remark}\label{R:indepSCF_notion}
Srour \cite[Theorem 13]{gS86} provided an explicit description of non-forking independence in SCF$_{p,\infty}$:  whenever the field extensions $k\subseteq K\subseteq \UU$ and $k\subseteq L\subseteq \UU$ are all separable, we have that
   \[ K\indscf_k L \ \iff \ K\indacf_k \ L \ \text{ and } \ K\indp_k L, \]
where $\indp$ denotes the ternary relation of $p$-independence in $\UU$.
\end{remark}

Algebraic indepen\-dence and $p$-independence in $\UU$ need not imply one another, even if all field extensions in question are separable: Consider $K=\F_p(x)$ and $L=\F_p(x+y^p)$, with $x$ and $y$ in $\UU\setminus \UU^p$ algebraically independent over $\F_p$. The field $K$ is algebraically independent from $L$ over $\F_p$, yet $K$ is not $p$-independent from $L$ over $\F_p$ in $\UU$. Similarly, choosing an element $u$ in $\UU^p$ which is transcendental over $\F_p$, the field $K=\F_p(u)$ is $p$-independent from $L=\F_p(u)$ over $\F_p$ in $\UU$, yet $K$ is clearly not algebraically independent from $L$ over $\F_p$.
Indeed, the notion of $p$-independence in $\UU$ is vacuous for subfields of $\UU^p$.

We collect below some observations concerning $p$-independence, which will be needed in the sequel.
\begin{fact}\label{F:pindep_pbasis}
Consider subsets $A$ and $S$ of $K$ as well as a subset $B$ of $\UU$.
\begin{enumerate}[(a)]
\item The set $A$ is $p$-independent over $k$ in $K$ if and only if the collection of \emph{$p$-monomials in $A$} \[ a_{1}^{k_1}\dots a_{n}^{k_n} \text{ with $a_i$ in $A$ and } 0\leqslant k_i<p-1\]  is linearly independent over $K^p k$ \cite[\S 1.2]{zC99}.

\item  The extension $k\subseteq K$ is separable if and only if some (or equivalently, every) $p$-basis of $k$ remains $p$-independent in $K$.

Moreover, if the extension $k\subseteq K$ is separable, then a subset $C$ of $k$ is $p$-independent in $k$ over a subfield $k_0$ of $k$ if and only if $C$ is $p$-independent over $k_0$ in $K$ \cite[\S \ 1.10]{zC99}.

\item Given subfields $k_0\subseteq k$ and $K_0\subseteq K$; if both field extensions $k\subseteq K$ and $k_0\subseteq K_0$  are separable algebraic, then $K^pk_0 = K^{p}k$ and $K_0^p k = K^{p}k$. In particular, every $p$-basis of $k$ over $k_0$ is a $p$-basis of $K$ over $K_0$.

\item If $A$ is a $p$-basis of $K$ over $k$ and $B$ is algebraically independent over $K$, then
\begin{itemize}
  \item  the set $A$ is a $p$-basis of $K(B)$ over $k(B)$;

  \item the set $B$ is a $p$-basis of $K(B)$ over $K$.
\end{itemize}

\item If the extension $k\subseteq K$ is separable and $A$ is $p$-independent in $K$ over $k$, then $A$ is algebraically independent over $k$. Moreover, if $K$ is  finitely generated over $k$ and $A$ is a $p$-basis of $K$ over $k$, then $A$ is a separating transcendence basis of $K$ over $k$ \cite[\S \ 1.13 (1 \& 2)]{zC99} (see also Fact \ref{F:sep_indep} (b) and Remark \ref{R:sep_ind_p-indep} in this context).
\end{enumerate}
\end{fact}
Whilst the following results are well-known,  we have decided nonetheless to include short proofs for the sake of the presentation.

\begin{lemma}\label{L:pindep_pbasis}
Let $A$ be a subset of $K$ and assume that both field extensions $k\subseteq K$ and $k\subseteq  L$ are separable.
\begin{enumerate}[(a)]
    \item  We have the following implication: \[ K\indacf_k L \  \Longrightarrow \left\{ \begin{minipage}{7cm} $K\subseteq K L$ is separable and \\[1mm] $K$ is $p$-independent from $L$ over $k$ in $K L$.\end{minipage} \right.\]
The converse holds when $K$ is finitely generated over $k$.

\item Suppose that $L\subseteq L_1$ is a separable field extension such that $K\indacf_k L_1$. Then, all field extensions in the following diagram are separable:
\begin{center}
\begin{tikzcd}
K  \arrow[hookrightarrow]{r} & KL \arrow[hookrightarrow]{r} & K L_1
\\
k \arrow[hookrightarrow]{u}  \arrow[hookrightarrow]{r} & L \arrow[hookrightarrow]{u} \arrow[hookrightarrow]{r} & L_1 \arrow[hookrightarrow]{u}
\end{tikzcd}
\end{center}
In particular, the field extensions $K\subset KL\subset KL_1$ are both separable.
\item Suppose $K\subseteq \UU$ and $L\subseteq \UU$ are both separable. If $K\indp_k L $, then $K L\subseteq \UU$ is separable. Moreover, the converse holds if $K\indacf_k L$.
\end{enumerate}
\end{lemma}

\begin{proof}
For the proof of the implication in (a),  the separability part is exactly \cite[Fact 5]{gS86}. The $p$-independence part is shown in \cite[\S 1.14]{zC99} under a stronger assumption of linear disjointness. Our argument comes from the proof of \cite[Fact 5]{gS86}.

It follows directly from the definitions that, if $K_0$ is $p$-independent from $L$ over $k$ in $K_0M$ for any finitely generated over $k$ subextension $k\subseteq K_0\subseteq K$, then $K$ is $p$-independent from $L$ over $k$ in $KL$.
Thus, we may assume that $K$ is finitely generated over $k$.

Choose therefore some subset $A$ of $K$ which is $p$-independent over $k$ in $K L$. In particular, the subset $A$ is $p$-independent over $k$ in $K$ and thus it extends to a $p$-basis $A'$ of $K$ over $k$. Since the extension $k\subseteq K$ is separable, we have that $A'$ is a separable transcendence basis of $K$ over $k$ by Fact \ref{F:pindep_pbasis} (e). We need only show that $A'$ is $p$-independent over $L$ in $K L$. Now, the field $K$ is algebraically independent from $L$  over $k$, so $A'$ is algebraic independent over $L$. Fact \ref{F:pindep_pbasis} (d) yields that $A'$ is $p$-independent in $L(A')$ over $L$. Now, the extension $k(A')\subseteq K$ is separable algebraic and thus so is $L(A')\subseteq K L$. We conclude that $A'$ is $p$-independent in $K L$ over $L$, as desired.

For the converse statement in (a), assume that $K$ is finitely generated over $k$. It suffices to show that a separating transcendence basis $A$ of $K$ over $k$ is algebraically independent over $L$. Now, the extension $K\subseteq K L$ is separable, so $A$ is $p$-independent over $k$ in $K L$. Since $K$ is $p$-independent from $L$ over $k$ in $K L$, we have that $A$ is $p$-independent in $K L$ over $L$. Fact \ref{F:pindep_pbasis} (e) applied to the separable extension $K\subseteq K L$
yields that $A$ is is algebraically independent over $L$, as desired.

For (b), note that $k\subseteq L$ is separable as well. Since $K\indacf_kL_1$, we also have $K\indacf_k L$, so  $L\subseteq K L$ is separable by (a). Now the extension $L\subseteq L_1$ is separable, so the independence $K L\indacf_L L_1$ yields (by (a) again) that the field extension $K L\subseteq K L_1$ is again separable, as desired.
\\
For a proof of (c), choose $p$-bases $S_k$ of $k$ as well as $S_K$ of $K$ over $k$ and $S_L$ of $L$ over $k$. Since $K\subseteq \UU$, resp. $L\subseteq \UU$, is separable, we have that $S_K$, resp. $S_L$, is $p$-independent over $k$ in $\UU$. The $p$-independence of $K$ from $L$ over $k$ yields that $S_K$ is $p$-independent over $L$ in $\UU$. In particular, the subset $S_k\cup S_K\cup S_L$ of $K L$ is  $p$-independent in $\UU$. We need only show that it is a $p$-basis of $K L$, which follows immediately from the equality
\[ (K L)^p\left(S_k\cup S_K\cup S_L\right)=K^p\left(S_k\cup S_K\right)L^p\left(S_k\cup S_L\right)=K L.\]

If we now assume that $K\indacf_k L$ and $K L\subseteq \UU$ is separable, then part (a) yields that $K$ is $p$-independent from $L$ over $k$ in $K L$ and thus in $\UU$, as desired.
\end{proof}

\begin{remark}\label{R:counter}
Note that the converse of Lemma \ref{L:pindep_pbasis} (a) need not hold without the assumption of finite generation of $K$. Indeed, 
if $x$ is transcendental over $\F_p$,  then $K=\F_p(x)^{\alg}$  is $p$-independent from $L=\F_p(x)$  over $\F_p$ in $KL=K$, since $K$ is perfect.
However, the field $K$ is not algebraically independent from $L$ over $\F_p$. 
\end{remark}

\section{Differential fields in positive characteristic}\label{S:DCFp}

In contrast to the characteristic $0$ case, the model theory of differentially closed fields in positive characteristic remains mostly unexplored since the seminal works of Shelah \cite{sS73} and Wood \cite{cW73, cW74,  cW76} in the case of a single derivation. Several decades later, the first author generalized some of the existing results  to the context of derivations of powers of the Frobenius map~\cite{pK05}. More recently, the second author, together with Ino \cite{ILS23},
has considered variations of the theory DCF$_p$ of differentially closed fields of positive characteristic (in a single derivation).

We fix a natural number $m\geqslant 1$ and work in the context of differential fields equipped with $m$-many commuting derivations.
Following Kolchin's presentation~\cite{eK73}, we will denote the set of distinguished (commuting) derivations by $\Delta=\{\partial_1,\dots,\partial_m\}$. From now on, by a differential field we mean one of the form $(K,\Delta)$ where the $\partial_i$'s are pairwise commuting derivations on $K$. In what follows we will simply say that $K$ is a differential field, rather than $(K,\Delta)$; and similary for differential field extensions $K\subseteq L$.


The class of existentially closed differential fields of characteristic $p$ has been shown to be recursively axiomatisable by Pierce~\cite{dP14}. Hence, the theory of differential fields in characteristic $p$ has a model-companion. We will denote this theory by DCF$_{p,m}$ in the language $\LL_\Delta$ of differential rings. Models of DCF$_{p,m}$ are called differentially closed fields. We will work inside a sufficiently saturated model $(\UU, \Delta)$ of DCF$_{p,m}$. For every differential subfield $K$ of $\UU$, we have that $K^p$ is a subfield of the \emph{field of absolute constants}
\[ \CC_K=\{a\in K\mid \partial_i(a)=0 \text{ for } 1\leqslant i\leqslant m\}\]

A differential field $K$ is said to be \emph{differentially perfect} if every differential field extension is separable, or equivalently, if  $\CC_K=K^p$, as shown by Kolchin  \cite[II.3]{eK73}. It is worth noting that this equivalence can be derived from the fact that for every differential field extension $K\subseteq L$ we have $K\indld_{\CC_K} \CC_L$, see \cite[II.1]{eK73}. In fact, the latter linear disjointness implies that the differential field extension $K\subseteq L$ is separable if and only if $\CC_K\indld_{K^p} L^p$. Thus,  if $K\subseteq L$ is separable and $L$ is differentially perfect, then so is $K$.  On the other hand, if $K\subseteq L$ is separably algebraic and $K$ is differentially perfect, then so is $L$.

\begin{remark}\label{R:diffstr} Let $K$ be a differential field.
\begin{enumerate}[(a)]
\item The differential structure on $K$ extends uniquely to $K^{\sep}$ and $\CC_{K^{\sep}}=\CC_K^{\sep}$.
\item If the constant element $b$ of $\CC_K$ does not belong to $K^p$, then there is a unique extension to a differential structure on $K(b^{1/p})$ with $\partial(b^{1/p})=0$ for all $\partial$ in  $\Delta$. Thus, every differential field embeds into a differential field extension $L$ with $L$ differentially perfect.
\item Every differentially closed field is differentially perfect. Furthermore, the underlying field is separably closed of infinite degree of imperfection.
\end{enumerate}
\end{remark}
\begin{proof}
	For (a), the uniqueness of the extension part is already shown in \cite[II.2]{eK73}. The linear disjointness $K\indld_{\CC_K} \CC_{K^{\sep}}$ yields that $\CC_{K^{\sep}}=\CC_K^{\sep}$.

For (b),  it follows from \cite[VIII.5]{sL02} that all derivations $\partial$ can be simultaneously extended to $K(b^{1/p})$, since the condition  $\partial(b^{1/p})=0$ for all $\partial$ in $\Delta$ ensures that the extensions of the derivations will commute on $K(b^{1/p})$. The last statement follows by a standard chain construction.

For (c), assume that  $K$ is differentially closed. Part (b) yields that $K$ must be differentially perfect whilst (a) shows that $K$ is separably closed. The existential closedness of $K$ ensures  that for every nontrivial differential polynomial $f(x)$ with coefficients in $K$ there exists some element  $a$ in $K$ with $f(a)\neq 0$. Hence, the field extension $K^p= \CC_K\subseteq K$ has infinite degree, by \cite[II.1]{eK73}, as desired.
\end{proof}

Note that it follows from (c) above that a differential subfield $K$ of the differentially closed field $\UU$ is differentially perfect if and only if the field extension $K\subseteq \UU$ is separable.

While the theory DCF$_{p,m}$ does not admit quantifier elimination in the language $\mathcal L_\Delta$, it does once we expand by the $p$-th root function on constants. Namely, we expand $\LL_\Delta$ by the following unary function
\[ \begin{array}{rccl}
    r:& \UU & \to & \UU \\[1mm]
     & x &\mapsto & \begin{cases} x^{1/p}, \text{ if $x$ belongs to $\UU^p=\CC_\UU$} \\
     0, \text{ otherwise.}
    \end{cases}  \end{array}\]

    In the language $\LL_\Delta^r=\LL_\Delta\cup\{r\}$ consider therefore the theory DPF$_{p,m}^r$ of differentially perfect differential fields  with the natural interpretation for the unary function symbol $r$.

    \begin{fact}
    It follows from Kolchin's result \cite[II.2]{eK73} that	the universal theory DPF$_{p,m}^r$ has the amalgamation property, so its model-completion  DCF$_{p,m}^r$ has quantifier elimination. Completeness of DCF$_{p,m}^r$ follows immediately from the fact that the prime field $\F_p$ is a differentially perfect differential field.
    \end{fact}
 The fact above yields some immediate consequences.
    \begin{cor}\label{C:diffperf}
\begin{enumerate}[(a)]
\item Two differentially perfect differential subfields $K$ and $K'$ of $\UU$ have the same type (with respect to some fixed enumeration) if and only if they are $\LL_\Delta$-isomorphic. 
\item Every differentially perfect differential subfield $K$ of $\UU$ is definably closed (in the sense of the theory $\DCF_{p,m}$) and its model theoretic algebraic closure coincides with $K^{\sep}$, which is also its relative (field) algebraic closure within $\UU$. More generally, the definable closure $\dcl(A)$ of a subset $A$ of $\UU$ is the smallest differentially perfect differential subfield of $\UU$ containing $A$ and the model theoretic algebraic closure $\acl(A)$ equals  $\dcl(A)^{\sep}$.
\end{enumerate}
\end{cor}

Shelah has shown that the theory DCF$_p=$ DCF$_{p,1}$ is stable. Standard type-counting methods can be used to deduce that  DCF$_{p,m}$ is stable. However, we will give a self-contained proof of the stability of DCF$_{p,m}$ in Section \ref{S:nf} by providing a complete description of forking independence. For this,  we need the following algebraic notions and results.

\begin{definition}\label{D:diff_sep_indep}
  Consider a differential field extension $k\subseteq K$ and a subset $A$ of $K$. We say that $A$ is \emph{differentially independent}, resp. \emph{separably differentially independent}, over $k$ if the family (or rather, the sequence, for we do not allow repetitions)
 $$\left(\partial_i^j(A)\right)_{1\leqslant i\leqslant m, 0\leqslant j}$$
is algebraically independent, resp. separably independent,  over $k$. 
\end{definition}

\begin{fact}\cite[Corollary 5, \S \ II.9]{eK73}\label{F:kolchin}
Consider a differential field extension $k\subseteq K$ and a subset $A\subseteq K$ which is differentially separably independent over $k$. The field of constants $\CC_{\diff{k}{A}}$ equals $\diff{k}{A}^p \CC_k $, where $\diff{k}{A}$ denotes the differential subfield of $K$ generated by $A$ over $k$.
\end{fact}
Given a differential subfield $K$ of our differentially perfect ambient model $\UU$, notice that $\CC_K^{1/p}$ is a subset of $\UU$ containing $K$, for $\CC_K$ contains $K^p$.  However, the field $\CC_K^{1/p}$ need not be stable under the derivations, so the iterates $\partial^j_i(a)$ of elements of $\CC_K^{1/p}$ need not lie in $\CC_K^{1/p}$.

The following notion will be fundamental for our description of non-forking. We are not aware of any occurrence of this notion in the existing literature. The name we chose reflects a crucial property stated in Lemma \ref{L:differentially traplem}.

\begin{definition}\label{D:trim}
The differential subfield $K$ of $\UU$ is said to be \emph{differentially transcendentally almost perfect}, in short \emph{differentially trap}, if there exists a $p$-basis $A$ of $\CC_K^{1/p}$ over $K$ which is separably differentially independent over $K$.
\end{definition}

Whilst the elements of $A$ are algebraic over $K$, they are purely inseparable over $K$, so it makes sense to demand that they are  separably differentially independent over $K$.

Note that every differentially perfect field is differentially trap with $A$ the empty $p$-basis. Whilst differential perfectness is an absolute condition, independent of the embedding into $\UU$, being differentially trap is a notion that is relative to the differential structure of $\UU$
(to see this, it suffices to consider different subfields of $\CC_\UU$).
\begin{lemma}\label{L:differentially traplem}
Given a differentially trap differential subfield $K$ of $\UU$ with the corresponding $p$-basis $A$ of $\CC_K^{1/p}$ over $K$ (so  $A$ is separably differentially independent over $K$), the differential field  $\diff{K}{A}$ is differentially perfect and $\acl(K)=\diff{K}{A}^{\sep}$.
\end{lemma}
\begin{proof}
Since $A$ is a $p$-basis of $\CC_K^{1/p}$ over $K$, we have that
\[ \CC_K^{1/p}=\left(\CC_K^{1/p}\right)^p   K(A)=\CC_K  K(A)=K(A),\]
hence we obtain that $\CC_K$ equals $K(A)^p$. By Fact \ref{F:kolchin}, we deduce that
\[ \CC_{\diff{K}{A}}=\CC_K  \diff{K}{A}^p=K(A)^p \diff{K}{A}^p=\diff{K}{A}^p,\]
and thus the differential field $\diff{K}{A}$ is differentially perfect, as desired.

For the last statement, note that the extension $\diff{K}{A}\subseteq \UU$ is separable (since $\diff{K}{A}$ is differentially perfect), and thus so is $\diff{K}{A}^{\sep}\subseteq \UU$, by Fact \ref{F:pindep_pbasis} (b) and (c). The previous characterization of the model-theoretic algebraic closure in DCF$_{p,m}$ yields that
\[ \acl(K)\subseteq \acl\left(\diff{K}{A}^{\sep}\right)=\diff{K}{A}^{\sep} \subseteq \acl(K),\] using that $A$ is algebraic over $\CC_K$, and thus over $K$.
\end{proof}
\begin{remark}\label{R:differentially trap}
\begin{enumerate}[(a)]
\item A directed union of a system of differentially trap differential subfields of $\UU$ is again a differentially trap differential subfield.
\item Given a differential subfield $K$ of $\UU$ and a $p$-basis $A$ of $\CC_K^{1/p}$ over $K$, we have that $A$ is separably differentially independent over $K$ if and only if the family
\[(\partial_i(A))_{1\leqslant i\leqslant m}\]
is differentially independent over $K$.
\item The field $K$ is differentially trap if and only if every $p$-basis of $\CC_K^{1/p}$ over $K$ is separably differentially independent over $K$.
\item Suppose that $K\subseteq L$ is a separably algebraic extension of differential fields. Then, $K$ is differentially trap if and only if $L$ is.

\end{enumerate}
\end{remark}
\begin{proof}
The proof of (a) is standard, so we will not include it. For the proof of (b), notice first (see the proof of Lemma \ref{L:differentially traplem}) that
 $K(A)=\CC_K^{1/p}$. One direction follows immediately from Lemma \ref{L:sepalg}, for the $p$-independent set $A$ over $K$ is separably independent over $K$ in $K(A)=\CC_K^{1/p}$ by Remark \ref{R:sep_ind_p-indep}. Therefore, we need only show that, given a separably differentially independent $p$-basis $A$ of $\CC_K^{1/p}$ over $K$, the family \[  \widetilde{A}:= \bigcup_{\substack{1\leqslant i\leqslant m\\ 1 \leqslant j\in \N}}  \partial_i^j(A)\] is algebraically independent over $K$. In fact, we will show that it is algebraically independent over $\CC_K^{1/p}$ (which contains $K$). This algebraic independence  over $\CC_K^{1/p}$  follows from Fact \ref{F:sep_indep} (b) once we show that the extension $\CC_K^{1/p}\subseteq \diff{K}{A}=K(A\cup \widetilde{A})$ is separable.

The linear disjointness $K\indld_{\CC_K} \CC_\UU$ yields that $K\indld_{\CC_K} K\langle A\rangle^p$. By assumption, the set $A\cup \widetilde{A}$ is separably independent over $K$ and thus $p$-independent over $K$ in $K(A\cup \widetilde{A})= K\langle A\rangle$ by Remark \ref{R:sep_ind_p-indep}. In particular, the set $A$ is $p$-independent over $K$ in $K\langle A\rangle$, which yields by Fact \ref{F:pindep_pbasis} (a) that $\CC_K^{1/p}$ is linearly disjoint from $K\langle A\rangle^p K$ over $K$. Transitivity of linear disjointness gives now that \[\CC_K^{1/p}\indld_{\CC_K} \  K\langle A\rangle^p.\] Since $\CC_K=\left(\CC_K^{1/p}\right)^p$, we conclude that the extension $\CC_K^{1/p}\subseteq K\langle A\rangle$ is separable, as desired.

For the proof of (c), let $K$ be a differentially trap differential field  witnessed by the $p$-basis $A$. We need to show that every other $p$-basis $B$ of $\CC_K^{1/p}$ is separably differentially independent over $K$. By (b), it is enough to show that $\partial_1(B)\cup\dots\cup \partial_m(B)$ is differentially independent over $K$. Following the proof of \cite[Proposition 2.7]{ILS23}, it suffices to show that for any $b_1,\ldots, b_n$ in $B$ and $\ell\geqslant 0$, the family
\[ (\partial_i^j(b_k)\mid 1\leqslant i\leqslant m \ , \  1\leqslant j \leqslant \ell \ , \ 1\leqslant k\leqslant n )\]
is algebraically independent over $K$.

Since $A$ is also a $p$-basis of $\CC_K^{1/p}$ over $K$, there is (without loss of generality) some $s\geqslant n$ and elements $a_1,\ldots, a_s\in A$ such that \[ b_1,\ldots,b_n\in \left(\CC_K^{1/p}\right)^p  K(a_1,\dots,a_s)=K(a_1,\dots,a_s).\] Furthermore,  we may assume that $K(a_1,\dots,a_s)=K(b_1,\dots,b_n,a_{n+1},\dots,a_s)$,  so
\begin{multline*} K(\partial_i^j a_k \mid  1\leqslant i\le m \ ; \ 0\leqslant j\leqslant \ell \ ; \ 1\leqslant k\leqslant s)= \\ = K(\partial_i^j b_k,\partial_i^j a_r \mid 1\leqslant i\leqslant m\ ; \  0\leqslant j\leqslant \ell \ ; \  1\leqslant k\leqslant n \ ; \ n+1\leqslant r\leqslant s).\end{multline*} Since the $a_i$'s and $b_i$'s are purely inseparable over $K$, the transcendence degree on the left-hand side has the largest possible value, namely $\ell m s$. In particular, we deduce that
\[ (\partial_i^j(b_k)\mid 1\leqslant i\leqslant m \ , \  1\leqslant j \leqslant \ell \ , \ 1\leqslant k\leqslant n )\]
is algebraically independent over $K$, as desired.

For the proof of (d), the linear disjointness $K\indld_{\CC_K} \CC_L$ yields that the extension $\CC_K\subseteq \CC_L$ is again separably algebraic.  Fact \ref{F:pindep_pbasis} (c) above yields that every $p$-basis of $\CC_K^{1/p}$ over $K$ remains a $p$-basis of $\CC_L^{1/p}$ over $L$. Together with part (c) above, we conclude that $K$ is differentially trap if and only if $L$ is.
    \end{proof}

\section{Forking independence and stationarity in DCF$_{p,m}$}\label{S:nf}
As in the previous section, we work inside a sufficiently saturated model $(\UU,\partial)$ of the theory DCF$_{p,m}$. As we have mentioned already, mimicking the argument for the ordinary case DCF$_p$, the theory DCF$_{p,m}$ can be shown to be stable by a type-counting argument (see Remark~\ref{R:criterion_consequences}(1) in the Appendix), yet it is not superstable, for the infinite descending chain of definable subfields
\[ \UU\supsetneq \UU^p \supsetneq \UU^{p^2}\supsetneq\cdots  \] witnesses that the $U$-rank cannot be ordinal-valued. In addition, DCF$_{p,m}$ does not have elimination of imaginaries (see \cite[Remark 4.3]{mMcW95} or for a more recent and detailed account \cite[Lemma 4.11]{cB25}); in particular, the fact that types over acl-closed sets are stationary cannot be deduced merely from stability but we will see it as a consequence of the results in this section (namely, from our algebraic description of forking independence).

Now, the work of Kim and Pillay \cite[Theorem 4.2]{KP97} (or rather, the stable version thereof due to Harnik and Harrington \cite[Theorem 5.8]{HH84}) implies that for stability it suffices to exhibit a notion of independence satisfying the well-known properties of forking independence, including stationarity over elementary substructures. Our motivation for this paper is to provide a self-contained exposition of the stability of DCF$_{p,m}$ together with an algebraic description of forking independence over arbitrary base subsets, so we will avoid taking shortcuts such as the one provided in \cite[Proposition 1.4]{MP24}.


In the following fact we note that
a description of non-forking in DCF$_p$ \emph{over models} appeared implicitly in Shelah's work \cite{sS73}.

\begin{fact}\label{F:Shelah}
	Let $K$ and $L$ be differentially perfect differential subfields of $\UU$ both containing an elementary substructure $M$ of $\UU$.  If $K\indld_{M} L$, then $K L$ is differentially perfect \cite[proof of Theorem 9]{sS73}.
	In particular, if we denote by $\inddcf \ $ forking independence in the theory DCF$_{p, m}$, we have that \[ K\inddcf_{M} L \  \iff \ K\indld_{M} L \ \stackrel{\ref{F:sep_indep} (a)}{\iff} \  K\indacf_{M} L\ {\iff} \ K\indscf_{M} L, \]
    For the first equivalence, left-to-right follows from Chatzidakis's \cite[Theorem 3.5]{zC99}, while right-to-left follows as an application of \cite[Proposition 1.4]{MP24}. The last equivalence follows from Lemma \ref{L:pindep_pbasis} (c) and Srour's description of forking in SCF$_{p,\infty}$ (see Remark~\ref{R:indepSCF_notion}).
\end{fact}

In the following remark, by a \emph{Picard-Vessiot closed} differential field $K$ we mean one satisfying the following: For each positive integer $n$, if $A_1,\dots,A_m$ are in $Mat_n(K)$ and there exists a differential field extension $K\subseteq L$ having $s$-many $L$-linearly independent vectors in $L^n$ each a solution of
\[ \partial_i\begin{pmatrix} x_1 \\ \vdots \\ x_n \end{pmatrix}=A_i \begin{pmatrix} x_1 \\ \vdots \\ x_n \end{pmatrix}\quad \text{ for }i=1,\dots,m,\]
then there exist $s$-many $K$-linearly independent solutions in $K^n$.

\begin{remark}
    Fact \ref{F:Shelah} above holds more generally: one can weaken the assumption that $M$ is an elementary substructure to simply asking that $M$ is acl-closed and Picard-Vessiot closed. Indeed, a straightforward adaptation of \cite[Lemma 7.6]{MPZ20} (in particular of Claim 4 thereof) shows that, when $M$ is PV-closed, $K\indld_{M} L$ implies that $K L$ is differentially perfect. The equivalence of the different notions of independence can be seen as a consequence of our description of forking independence in Theorem~\ref{T:main}.
    \end{remark}

\begin{definition}\label{D:indstar}
Given differentially perfect differential subfields $k$, $K$ and $L$ of $\UU$ with $k\subseteq K\cap L$, set \[ K\ind^*_k L \ \iff \ \begin{cases} K\indacf\limits_k L, \text{ and }\\[1mm]
 \text{the compositum $KL$ is differentially trap}
\end{cases}.\]
More generally, given subsets $A$, $B$ and $C$ of $\UU$ with $C\subseteq A\cap B$, set $A\ind^*_C B$ if $\dcl(A)\ind^*_{\dcl(C)} \dcl(B)$, for the definable closure $\dcl(D)$ of a set $D$ is the smallest differentially perfect differential subfield of $\UU$ containing $\F_p(D)$ (see Corollary~\ref{C:diffperf}~(b)).
\end{definition}

\begin{remark}\label{R:remark_indstar}
Let us fix differentially perfect differential subfields $k$, $K$ and $L$ of $\UU$ with $k\subseteq K\cap L$.
\begin{enumerate}[(a)]
  \item By Remark \ref{R:differentially trap} (d), we have the following
$$K^{\sep}\ind^{*}_{k^{\sep}} L^{\sep}\quad  \iff\quad  K\ind^{*}_{k} L.$$
Therefore, using Corollary \ref{C:diffperf} (b), we obtain
$$A\ind^{*}_C B\quad  \iff\quad  \acl(A)\ind^{*}_{\acl(C)} \acl(B)$$
for any $A,B$ and $C$ as in Definition \ref{D:indstar}.

\item If the compositum $KL$ is differentially perfect, we have the following equivalence \[ K\ind^*_k L \ \iff \ K\indacf_k L,\]
since differentially perfect fields are differentially trap. In particular, we recover from Fact \ref{F:Shelah} that $\ind^*$ coincides with $\indacf$ for differentially perfect fields when the base subfield is an  elementary substructure of $\UU$ (and a posteriori algebraic independence coincides with forking independence for differentially perfect fields over elementary substructures).

\item Suppose $K\ind^{\ACF}_kL$. Then, the compositum field $KL$ is differentially perfect (and is hence differentially trap) if and only if $K\ind^{\SCF}_k L$, by Lemma \ref{L:pindep_pbasis} (c). Therefore, we have the following implications:
\[K\ind^{\SCF}_kL\quad \Longrightarrow\quad K\ind^{*}_kL\quad \Longrightarrow \quad  K\ind^{\ACF}_kL. \]
To see that none of these implications can be reversed, consider $x,y\in \UU$ algebraically independent over $\F_p$ and such that $\partial(x)=1$ (for simplicity, we assume $m=1$ here). Then, we take
$$k=\F_p,\quad K=\F_p(x),\quad L=\F_p(x+y^p).$$
Clearly, $K\nind^{\SCF}_k L$ and $K\ind^{\ACF}_k L$, but $KL$ has differential trap if and only if $y$ is separably differentially transcendental over $\F_p(x,y^p)$ (see Example \ref{exampled1} for some extra details).

\item Just as $\ind^{\SCF}$ involves two generally unrelated conditions (namely, algebraic independence and $p$-independence, see Remark~\ref{R:indepSCF_notion}), for $\ind^{*}$ we need both of the conditions in Definition~\ref{D:indstar}; i.e., the field independence
$K\indacf\limits_k L$ does not imply that $KL$ is differentially trap and the opposite implication need not hold either. This can be seen using the example from (c) above and considering e.g. subfields of $\UU^{p^{\infty}}$.
\end{enumerate}
\end{remark}


\begin{theorem}\label{T:main}
The ternary relation $\ind^*$ defined above is invariant under automorphisms and symmetric. It satisfies local and finite character as well as extension and transitivity/monotonicity. Furthermore, it satisfies  stationarity whenever the base set is an $\acl$-closed subfield (in the model theoretic sense) of $\UU$.

In particular, the theory $\DCF_{p, m}$ is stable and the relation $\ind^*$ coincides with forking independence.  Moreover, types over $\acl$-closed subsets of $\UU$ are stationary.
\end{theorem}
\begin{proof}
	
	Throughout this proof, let as assume that $k, K, L$ and $L_1$ are differential subfields of $\UU$ such that $k\subseteq K$ and $k\subseteq L\subseteq L_1$. It is enough to check our independence conditions for acl-closed sets only, therefore (using Remark \ref{R:remark_indstar} (a)) we will assume that $k, K, L$ and $L_1$ are algebraically closed subfields of $\UU$ (in the model theoretic sense). In particular, all fields extensions among $k,K,L,L_1$ and $\UU$ are regular. Types are always complete and taken in the sense of the theory DCF$_{p,m}$, unless explicitly stated.

Invariance and symmetry of $\ind^*$ follow immediately from the definition of $\ind^*$.

\begin{claim}\label{Cl:ext}
The relation $\ind^*$ satisfies extension: Given some fixed enumeration of $K$, there exists a differential subfield $K'$ of $\UU$ with \[ K'\ind^*_k L\] such that $K'$ and $K$ have the same type over $k$ (with respect to the chosen enumeration).
\end{claim}
\begin{claimproof}
The algebraically closed field $\UU^{\alg}$ is sufficiently saturated, so  there exists a $k$-isomorphic  copy $K_1$ of $K$ in  $\UU^{\alg}$ with $K_1\indacf_k L$. Transporting the differential structure from $K$ to $K_1$, we may assume that $K_1$ is $k$-isomorphic to $K$ as differential fields.

The extension $k\subseteq K_1$ is regular (for $k\subseteq K$ is), so $K_1\indld_kL$ by Fact \ref{F:sep_indep} (a). In particular, using that
\[K_1 L\cong_k\mathrm{Frac}\left(K_1\otimes_kL\right),\]
we may equip $K_1 L$ with commuting derivations $\widetilde\Delta=\{\widetilde\partial_1,\dots,\widetilde\partial_m\}$ 
extending both the differential structures of $K_1$ and $L$. Choose now a $p$-basis $A_1$ of $\CC_{K_1L}^{1/p}$ over $K_1L$. By \cite[Chapter IV, Theorem 17]{nJ12}, there exists a differential field $(F,\widetilde\partial_1,\dots,\widetilde\partial_m)$ which is a differential field extension of $K_1L$, whose cardinality $|F|$ is small with respect to the saturation of $\UU$, such that $\CC^{1/p}_{K_1L}\subseteq F$ and \[ \widetilde\partial_1(A_1)\cup\cdots \cup \widetilde\partial_m(A_1)\] is differentially independent over $K_1L$. By Remark \ref{R:diffstr}, we may assume that $F$ is differentially perfect.  Using the saturation of our existentially closed model $\UU$, we can now embed $F$ over $L$ in $\UU$ via $\Psi:M \hookrightarrow \UU$. Setting $K'=\Psi(K_1)$, notice that the differential subfield $K'$ of $\UU$ is differentially perfect and isomorphic to $K$ over $k$, so they have the same type over $k$ by Corollary \ref{C:diffperf} (a).

In order to show that $K'\ind^*_k L$, notice that $K'$ and $L$ are linearly disjoint, and thus algebraically independent, over $k$, by construction. Thus, we need only show that $K'L$ is differentially trap, which follows immediately from the fact that the set $A'=\Psi(A_1)$ is a $p$-basis of $\CC_{K'L}$ over $K'L$ with $\partial_1(A')\cup\cdots \cup \partial_m(A')$ differentially independent over $K'L$.
\end{claimproof}

\begin{claim}\label{Cl:stat}
	The relation $\ind^*$ satisfies stationarity over
    $\acl$-closed differential subfields of $\UU$:  Given some enumeration of $K$ and a $k$-isomorphic copy $K'$ in $\UU$ such that
	\[ K\ind^*_k L \text{ and  } K'\ind^*_k L,\]
	then $K'$ and $K$ have the same type over $L$.
\end{claim}
\begin{claimproof}
	This proof goes along the same lines as the proof of extension in Claim \ref{Cl:ext}, so we will be succinct. Without loss of generality, all the field extensions considered are regular, so there exists  a differential $L$-isomorphism $\Psi$ mapping $KL$ to $K'L$ with $\Psi(K)=K'$, by a standard tensor product argument.
	
	Both fields $KL$ and $K'L$ are differentially trap,  so choose a $p$-basis $A$ of $\CC_{KL}^{1/p}$ over $KL$ with $\partial_1(A)\cup\cdots \cup \partial_m(A)$ differentially independent over $KL$. Thus, the map $\Psi$ extends to a differential $L$-isomorphism \[ \Psi':\diff{KL}{A} \to\diff{K'L}{A'}\] with $A'=\Psi'(A)$ a $p$-basis of $\CC_{K'L}^{1/p}$ over $K'L$ .  By Lemma \ref{L:differentially traplem} and Remark \ref{R:differentially trap} (c), both differential fields $\diff{KL}{A}$ and $\diff{K'L}{A'}$ are differentially perfect. These two fields are isomorphic via $\Psi'$ (which maps $K$ onto $K'$), so we deduce that $K$ and $K'$ have the same type over $L$, by Corollary \ref{C:diffperf} (a), as desired.
\end{claimproof}

Recall that $k,K,L$ and $L_1$ are  differential subfields of $\UU$ such that $k\subseteq K$ and $k\subseteq L\subseteq L_1$ with all field extensions regular (hence, separable). In particular, algebraic independence coincides with linear disjointness by Fact \ref{F:sep_indep} (a).

\begin{claim}\label{Cl:bases_union}
If $K\indld_k L_1$, then every subset $S$ of $\CC_{KL}^{1/p}$ which is $p$-independent in $\CC_{KL}^{1/p}$ over $KL$ remains $p$-independent in $\CC_{KL_1}^{1/p}$ over $KL_1$.

In particular, there exists a $p$-basis of $\CC_{KL_1}^{1/p}$ over $KL_1$ of the form $S_0\cup S_1$, where $S_0$ is a $p$-basis of $\CC_{KL}^{1/p}$  over $KL$ and $S_1$ is $p$-independent over $KL(S_0)$. In particular, the set $S_0$ lies in $\acl(K\cup L)$.
\end{claim}
\begin{claimproof}
By Lemma \ref{L:pindep_pbasis} (b),  the following commutative diagram consists of separable field extensions:
\begin{center}
	\begin{tikzcd}
		K  \arrow[hookrightarrow]{r} & KL \arrow[hookrightarrow]{r} & KL_1
		\\
		k \arrow[hookrightarrow]{u}  \arrow[hookrightarrow]{r} & L \arrow[hookrightarrow]{u} \arrow[hookrightarrow]{r} & L_1 \arrow[hookrightarrow]{u}
	\end{tikzcd}
\end{center}
	Now, since $KL\subseteq KL_1$ is separable, we have that $KL\indld_{(KL)^p} (KL_1)^p$. Using that Frobenius is a field monomorphism, we deduce that $(KL)^{1/p}\indld_{KL} KL_1$, so $\CC^{1/p}_{KL}\indld_{KL_1} KL$. By Fact \ref{F:pindep_pbasis} (a), the $p$-monomials on $S$ are linearly disjoint over $KL$ and thus over $KL_1$, so $S$ as a subset of $\CC_{KL_1}^{1/p}$ is $p$-independent over $KL_1$.
	
	The last statement follows immediately from the fact that $p$-closure is a pregeometry.
\end{claimproof}

\begin{claim}\label{Cl:mono}
The relation $\ind^*$ satisfies monotonicity: If $K\ind^*_k L_1$, then $K\ind^*_k L$.

In particular, the relation $\ind^*$ satisfies finite character: \[ K\ind^*_k L_1 \ \iff \ \diff{k}{\bar a}\ind^*_k \diff{k}{\bar b} \text{ for all tuples  $\bar a$ in $K$ and $\bar b$ in $L_1$}.\]
\end{claim}
\begin{claimproof}
	Since algebraic independence satisfies monotonicity, we need only show that $KL$ is differentially trap. The independence $K\ind^*_k L_1$ yields that  $K\indacf_k \ L_1$, and thus  $K\indld_kL_1$ by Fact \ref{F:sep_indep} (a). Now, the field $KL_1$ is differentially trap, so by Remark \ref{R:differentially trap} (c), the basis $S_0\cup S_1$ of $\CC_{KL_1}^{1/p}$  as in Claim \ref{Cl:bases_union} is separably differentially independent over $KL_1$. In particular the $p$-basis $S_0$ of $\CC_{KL}^{1/p}$ is separably differentially independent over $KL_1$ and thus over $KL$. We conclude that  $KL$ is differentially trap, as desired.
	
Finite character now follows from Symmetry, Monotonicity and Remark \ref{R:differentially trap} (a).
\end{claimproof}

\begin{claim}\label{Cl:locchar}
	The relation $\ind^*$ satisfies local character: If $K$ is countable, then there exists some countable differentially perfect differential subfield $\ell$ of $L$ with $K \ell\ind^*_{\ell} L$.
\end{claim}
\begin{claimproof}
By Claim \ref{Cl:mono} and Remark \ref{R:remark_indstar} (c), it is enough to find $\ell$ as above and a differentially perfect subfield $K'$ of $\UU$ containing $K\ell$ such that $K'\ind^{\SCF}_{\ell} L$. To do that, we follow the proof of \cite[Fact 2.3]{pK05}.
	
By the local character of $\ind^{\SCF}$, there is a countable subfield $\ell_0\subseteq L$ such that $K\ell_0\ind_{\ell_0}L$. Set now $K_1=\dcl(\ell_0\cup K)$ the smallest differentially perfect differential subfield of $\UU$ containing $\ell_0\cup K$, by Corollary \ref{C:diffperf} (b). Notice that $K_1$  is again countable.

By the local character of $\ind^{\SCF}$ again, where we name constants for the elements of the differentially perfect differential subfield $\dcl(\ell_0)$ of $K_1\cap L$, there is a countable subfield $\ell_1\subseteq L$ such that $K_1\ind_{\ell_1}L$ and $\dcl(\ell_0)\subseteq \ell_1$. We iterate and define inductively an increasing sequence of countable subfields
$$\ell_0\subseteq \ell_1\subseteq  \ldots \subseteq L,$$ such that the differentially perfect differential subfield $\dcl(\ell_n)$ belongs to $\ell_{n+1}$ as well as an increasing sequence of differentially perfect differential subfields
$$K=K_0\subseteq K_1\subseteq K_2\subseteq \ldots\subseteq \UU$$
such that
\[ \dcl(\ell_n)\subseteq K_{n+1} \text{ and } \ K_n\ind^{\SCF}_{\ell_n}L \ \text{for all $n$ in $\N$}.\]

The differential fields
\[K'=\bigcup_n K_n \ \text{ and } \ \ell=\bigcup_n \ell_n\] are differentially perfect, by construction. By monotonicity and local character of $\ind^{\SCF}$, we have that $K'\indscf_{\ell} L$, as desired.
\end{claimproof}

For the next two claims, we will need the extra assumption that the compositum $KL$ is a differentially trap differential field, which will always be the case whenever $K\ind^*_k L_1$, so in particular if $K\ind^*_k L$ by Claim \ref{Cl:mono}. Note that the compositum $KL$ \emph{need not} be differentially perfect, so in the proof of transitivity, we will need to work instead with its definable closure $\dcl(K\cup L)$, which is differentially perfect by Corollary \ref{C:diffperf} (b).
\begin{claim}\label{Cl:baseS1}
If the differential field $KL$ is differentially trap and $S_0\cup S_1$ is a $p$-basis of $\CC^{1/p}_{KL_1}$ over $KL_1$ as in Claim \ref{Cl:bases_union}, then $S_1$ is a $p$-basis of $\CC^{1/p}_{\dcl(K\cup L)L_1}$ over $\dcl(K\cup L)L_1$.
\end{claim}
\begin{claimproof}
Since $S_0\cupdot S_1$ is a $p$-basis of $\CC^{1/p}_{KL_1}$ over $KL_1$, we get that $S_1$ is a $p$-basis of $\CC^{1/p}_{KL_1}$ over $KL_1(S_0)$. Moreover, the set \[ \widetilde{S_0} := \bigcup_{\substack{1\leqslant i\leqslant m\\ 1 \leqslant j\in \N}}  \partial_i^j(S_0) \] is algebraically independent over $\CC^{1/p}_{KL_1}$ and
\[ KL_1\langle S_0\rangle=KL_1(S_0)(\widetilde{S_0}).\]
Using Fact \ref{F:pindep_pbasis} (d), we obtain that $S_1$ is a $p$-basis of $\CC^{1/p}_{KL_1}KL_1\langle S_0\rangle$ over $KL_1\langle S_0\rangle$. By Lemma \ref{L:differentially traplem}, we get
\[\dcl(K\cup L)L_1=KL\langle S_0\rangle L_1=KL_1\langle S_0\rangle.\]
To finish the proof using Fact \ref{F:pindep_pbasis} (c), we need to notice that
\[ \CC^{1/p}_{KL_1}KL_1\langle S_0\rangle \subseteq \CC^{1/p}_{KL_1\langle S_0\rangle} \subseteq
\left(\CC^{1/p}_{KL_1}KL_1\langle S_0\rangle\right)^{\sep}.\]
The first inclusion is immediate, so we need only show the second one. Remark~\ref{R:diffstr}(a) yields that
\begin{IEEEeqnarray*}{rCl}
\CC_{KL_1\langle S_0\rangle} & \subseteq & \CC_{\left(KL_1\langle S_0\rangle\right)^{\sep}}\\
 & = & \left(\CC_{KL_1\langle S_0\rangle}\right)^{\sep}\\
 & = & \left(\CC_{KL_1}\left(KL_1\langle S_0\rangle\right)^p\right)^{\sep},
\end{IEEEeqnarray*}
so the desired inclusion follows after composing with  the inverse of the Frobenius map, which is injective.
\end{claimproof}

We have now all the ingredients to establish the transitivity
$\ind^*$.
\begin{claim}\label{Cl:transit}
The relation $\ind^*$ satisfies transitivity:
\[K\ind^{*}_k L\ \ \text{and}\ \ \dcl(K\cup L)\ind^{*}_L L_1 \ \ \ \ \   \iff\ \ \ \ \  K\ind^{*}_k L_1.\]
\end{claim}
\begin{claimproof}
For ($\Leftarrow$),  assume that $K\ind^{*}_k L_1$.  Claim \ref{Cl:mono} yields that $K\ind^{*}_k L$, so we need only show that  $K\ind^{*}_L L_1$, or equivalently, that $\dcl(K\cup L)\ind^*_L L_1$, for $KL$ need not be differentially perfect (but its definably closure $\dcl(K\cup L)$ is).

We first show that  $\dcl(KL)\ind^{\ACF}_L L_1$. Since the field
 $KL$ is differentially trap, by Claim \ref{Cl:mono}, it follows from  Lemma \ref{L:differentially traplem} (and the fact that $S_0\subseteq (KL)^{\alg}$) that  \begin{equation}
  \tag{$1$}
\dcl(K\cup L)=KL(\widetilde{S_0}),
\label{one}\end{equation} where \[ \widetilde{S_0}:=\bigcup_{\substack{1\leqslant i\leqslant m\\ 1 \leqslant j\in \N}}  \partial_i^j(S_0)).\]

Now, because the set $(\partial_i(S_0\cup S_1))_{1\leq i\leq m}$ is differentially independent over $KL_1$, then so is $(\partial_i(S_0))_{1\leq i\leq m}$.
In particular, we obtain the following independence
\begin{equation}
  \tag{$2$}
\widetilde{S_0}\ind^{\ACF}_{KL} KL_1.
\label{two}
\end{equation}
Combining (\ref{one}) and (\ref{two}), we deduce that
\begin{equation}
  \tag{$3$}
\dcl(K\cup L)\ind^{\ACF}_{KL} L_1.
\label{three}
\end{equation}
Since $K\ind^{\ACF}_k L_1$ and $k\subseteq L\subseteq L_1$, we have that $KL\ind^{\ACF}_L L_1$, which yields together with (\ref{three}) that $\dcl(K\cup L)\ind^{\ACF}_L L_1$.

Let us now conclude this direction showing that $\dcl(K\cup L)L_1$ is differentially trap. By Claim \ref{Cl:baseS1},  it suffices to show that $(\partial_i(S_1))_{1\leq i\leq m}$ is differentially independent over $\dcl(K\cup L)L_1$.
We know that
$(\partial_i(S_0))_{1\leq i\leq m}\cup (\partial_i(S_1))_{1\leq i\leq m}$ is differentially independent over $KL_1$. Therefore, the set $(\partial_i(S_1))_{1\leq i\leq m}$ is differentially independent over \[ KL_1\langle (\partial_i(S_0))_{1\leq i\leq m}\rangle= KL(\widetilde{S_0})L_1.\] Now, by the previous discussion, we have that $\dcl(K\cup L)=KL(\widetilde{S_0})$, so $(\partial_i(S_1))_{1\leq i\leq m}$ is differentially independent over $\dcl(K\cup L)L_1$, as desired.

For the direction ($\Rightarrow$), assume both independences $K\ind^{*}_k L$ and $K\ind^{*}_L L_1$. In particular, we have the field independences \[ K\indacf_k L \ \text{ and } \ KL\indacf_L L_1.\]
Transitivity of algebraic independence $\ind^{\ACF}$ yields that $K\ind^{\ACF}_k L_1$, so we need only show that $KL_1$ is differentially trap, assuming that $KL$ is, for $K\ind^*_k L$.

By Claim \ref{Cl:baseS1}, the set $S_1$ is a $p$-basis of $\CC^{1/p}_{\dcl(KL)L_1}$ over 
$\dcl(K\cup L)L_1$. Observe that the differential field $\dcl(K\cup L)L_1$ is trap, by the independence \[ \dcl(K\cup L)\ind^*_L L_1,\]
so the
set $(\partial_i(S_1))_{1\leq i\leq m}$ is differentially independent over $\dcl(K\cup L)L_1$. Now, the elements of the $p$-basis $S_0$ of $\CC_{KL}^{1/p}$ (and their derivatives) are clearly contained in $\dcl(K\cup L)$, so $(\partial_i(S_1))_{1\leq i\leq m}$ is differentially independent over $\diff{KL_1}{S_0}$. Summarizing, there exists by Claim \ref{Cl:bases_union} a $p$-basis $S_0\cup S_1$ of $\CC_{KL_1}^{1/p}$ such that
\[(\partial_i(S_0))_{1\leq i\leq m}\cup (\partial_i(S_1))_{1\leq i\leq m}\]
is differentially  independent over $KL_1$, and thus the differential field $KL_1$ is trap, as desired.
\end{claimproof}
The stability of DCF$_{p,m}$ and the description of non-forking independence follows now from Claims \ref{Cl:ext}, \ref{Cl:stat}, \ref{Cl:mono}, \ref{Cl:locchar} and \ref{Cl:transit}.
\end{proof}

We finish this section with some considerations on the behaviour of forking independence on the set of solutions of the equation $\partial(x)=1$ in the theory DCF$_p$=DCF$_{p,1}$. This study was the starting point of our research and we believe it may bring some intuition to our  notion of differentially trap fields.
\begin{example}\label{exampled1}
Working inside a sufficiently saturated model $\UU$ of the theory DCF$_p$,
every solution $a$ to the equation $\partial(x)=1$ is transcendental over the perfect field $\F_p$. Moreover, the differential field $\diff{\F_p}{a}=\F_p(a)$ is differentially perfect since $\partial(a)=1$.
If $b$ is any other solution, the element $b-a$ belongs to $\CC_\UU$ and thus we can consider $\lambda=(b-a)^{1/p}$ in $\CC_{\F_p(a)\F_p(b)}^{1/p}$. Since $b=a+\lambda^p$, we have \[\F_p(a)\F_p(b)=\F_p(a, b)=\F_p(a, \lambda^p),\] so $\lambda$ generates $\CC_{\F_p(a, b)}^{1/p}=\F_p(a, \lambda)$ over $\F_p(a,b)$.

Assume now that $a$ and $b$ are $\inddcf$-independent, so in particular $a$ and $b$ (and thus $a$ and $\lambda$) are algebraically independent over $\F_p$. It follows that $\{\lambda\}$ is a $p$-basis of $\CC_{\F_p(a, b)}^{1/p}$ over $\F_p(a,b)$, since \[ \lambda \notin \left( \CC_{\F_p(a, b)}^{1/p} \right)^p \F_p(a, b)=\F_p(a,b)=\F_p(a,\lambda^p).\] Our description of $\ind^{\DCF}$ from Theorem \ref{T:main} over the differentially perfect differential subfield $\F_p$ implies that $\partial(\lambda)$ must be differentially transcendental over $\F_p(a,b)$. Using that differential algebraicity induces a pregeometry (see \cite[Proposition 8 in \S II.8]{eK73}),
it is easy to see that this latter condition is equivalent to  $\lambda$ itself being differentially transcendental over $\F_p$, since both $a$ and $b$ are differentially algebraic.

Thus, it follows from Theorem \ref{T:main} that for $a, b$ solutions of $\partial(x)=1$ \[ a\inddcf b \ \iff \ \begin{cases} a\indacf b \\[1mm]
\text{ and } \\[1mm]
   (b-a)^{1/p} \text{ is differentially transcendental over } \F_p
\end{cases}  \]
\end{example}
The above example highlights a fundamental heuristic difference between forking independence in DCF$_{p,m}$ with respect to other well-known theories of existentially closed fields with operators such as DCF$_0$ or ACFA, which are \emph{$1$-based} over the reduct to the pure algebraically closed field, in the sense of \cite[Definition~4.1]{BMPW15}.  However, if there was a description of forking independence in DCF$_p$ in terms of the reduct to the theory of separably closed fields, Example \ref{exampled1} would yield that differential transcendence over $\mathbb{F}_p$ reduces to a condition in the pure separably closed field, which is not the case.

\section{Bernoulli differential equations in positive characteristic}\label{S:Ber}
An ordinary differential equation of the form
 \[\partial(T)=T^n \quad \text{ for } n\in \mathbb Z\]
 is a special case of the \emph{Bernoulli differential equation} \cite{bernoulli}. As shown by Leibniz, the above equation reduces to the linear differential equation $\partial(X)=1$ via the following variable substitution (for $n\neq 1$):
\[X= \frac{1}{(1-n)}T^{1-n}.\]
In the characteristic 0 case (i.e., in the theory DCF$_0$), it follows that the set of solutions of the equation $\partial(T)=T^n$ is \emph{almost internal to the constant field}: working inside a sufficiently saturated differentially closed field $\mathcal M$ of characteristic $0$,  there exists a countable  differential subfield $k$ such that every solution of the Bernoulli equation is algebraic over the compositum subfield $k\CC_{\mathcal M}$.

In the positive characteristic case, Leibniz's method works when $p$ does not divide $n-1$  and gives a finite-to-one definable function from the set of solutions of the equation $\partial(T)=T^n$ onto $\UU$. Note that when $n=1$ we obtain a similar result: the set of solutions of $\partial(T)=T$ is in definable bijection with $\CC_\UU=\UU^p$, and hence also with $\UU$ (after composing with the inverse of Frobenius).


Now, if $n=pm+1$ for some $m\neq 0$, then Leibniz's method no longer works for obvious reasons and in fact the solution set of $\partial(T)=T^n$ is not in definable finite-to-finite correspondence with $\UU$. In this section we will exhibit an unexpected model-theoretic explanation of this. We start with a couple of general remarks.

\begin{remark}\label{R:isol}
Consider the equation $\partial(T)=T^n$ with $n$ in $\mathbb Z$.
\begin{enumerate}[(a)]
\item
    The formula
    \[\partial(T)=T^n\land T\ne 0\]
    isolates a complete stationary type over $\F_p$: Indeed, if $a\ne 0$ and $\partial(a)=a^n$, then $a\notin \UU^p$, for its derivative is not zero. In particular, the differential field $\diff{\F_p^{\alg}}{a}=\F_p^{\alg}(a)$ is differentially perfect, for the extension $\F_p^{\alg}(a)\subseteq \UU^p$ is separable by Fact \ref{F:pindep_pbasis} (b). Corollary \ref{C:diffperf} (a) yields now that any two non-zero solutions of the equation have the same type over the algebraically closed differential field $\F_p^{\alg}$ and hence the type $tp(a/\mathbb F_p)$ is stationary by Theorem~\ref{T:main}.

    \item If $n=p^km+1$ where $k>0$, $p$ does not divide $m$ and $\partial(a)=a^n$, then the element $b=ma^m$ is a solution of the equation $\partial(X)=X^{p^k+1}$. Since $\UU$ is separably closed, the map $x\mapsto mx^m$ is a surjective map with finite fibers (of size $m$) from the set of solutions in of the equation $\partial(T)=T^n$ onto the set of solutions of $\partial(X)=X^{p^k+1}$. Since we later show that the set of non-zero solutions of $\partial(X)=X^{p^k+1}$ is geometrically trivial and $\aleph_0$-categorical (having the induced structure of a pure set, see Theorem \ref{T:type_trivial}), we get that the set of non-zero solutions of $\partial(T)=T^{n}$  is geometrically trivial and $\aleph_0$-categorical as well in this case.
\end{enumerate}
\end{remark}
In view of Remark \ref{R:isol}(b), we will now focus our attention on studying the behaviour of the generic solutions of the equation $\partial(T)=T^{p^k+1}$ for $k>0$ from the point of view of geometric stability.
For the following proposition, we first need an auxiliary result, which is immediate.
\begin{remark}\label{R:p-mon_diff}
Consider a $p$-monomial $w=X_1^{\alpha_1}\ldots X_n^{\alpha_n}$ with $0\leqslant \alpha_i\leqslant p-1$ for all $1\leqslant i\leqslant n$.
Given
$\bar a=(a_1,\dots,a_n)\in \UU^n$, we have that
\[a_i\frac{\partial w}{\partial X_i}\left(\bar{a}\right)=\alpha_i\, w\left(\bar{a} \right) \text{ for each } 1\leqslant i\leqslant n,\]
where on the right-hand side all the $\alpha_i$'s are considered as elements of $\F_p$. In particular,  if each $a_i$ is a solution of $\partial(T)=T^{p^{k_i}+1}$ for some $k_i\geqslant 1$, then \[ \partial(w(\bar a))=\sum_{i}^{n} \frac{\partial w}{\partial X_i}(\bar a) \partial(a_i) = \sum_{i}^{n} \frac{\partial w}{\partial X_i}(\bar a) a_i^{p^{k_i}+1}=  \left(\sum\limits_{i=1}^n \alpha_i a_i^{p^{k_i-1}} \right)^p w(\bar a).\]
\end{remark}

\begin{prop}\label{P:trivial_new}
Consider pairwise distinct elements $a_1,\ldots,a_n$, each one a non-zero solution of the equation
$\partial(X)=X^{p^{k_i}+1}$ for integers $k_i\ge 1$. Then, the elements $a_1,\ldots,a_n$ are $p$-independent over $\mathbb{F}_p$ in $\UU$ and thus they are  algebraically independent over $\F_p$, by Fact \ref{F:pindep_pbasis} (e).
Furthermore, the differential field  \[ \diff{\F_p}{a_1,\ldots, a_n}=\F_p(a_1,\ldots, a_n)\] is differentially perfect, by Lemma \ref{L:pindep_pbasis} (c) and Remark \ref{R:isol} (a).
\end{prop}
\begin{proof}
We first start with an easy observation.
\begin{claim*}
The elements $1,a_1,\ldots, a_n$ are linearly independent over $\UU^p$.
\end{claim*}
\begin{claimproofstar}
Suppose not and we will observe first that then $a_1,\ldots, a_n$ are linearly dependent over $\UU^p$. Take $\lambda_0,\lambda_1,\ldots,\lambda_n$ in $\UU$ which are not all $0$ and such that
\[\lambda_0^p+\lambda_1^p a_1+\cdots+\lambda_n^p a_n=0.\]
It is clear that then not all $\lambda_1,\ldots,\lambda_n$ are $0$ and we have
\[ 0= \partial(-\lambda_0^p)=\partial \left( \sum_{i=1}^n \lambda_i^p a_i\right) = \sum_{i=1}^n \lambda_i^p \partial(a_i)= \sum_{i=1}^n \lambda_i^p a_i^{p^{k_i}+1}= \sum_{i=1}^n \left(\lambda_i a_i^{p^{k_i-1}}\right)^p a_i.\]
If $1\leqslant i\leqslant n$ is such that $\lambda_i\neq 0$, then $\lambda_i a_i^{p^{k_i-1}}\neq 0$ as well, so $a_1,\ldots, a_n$ are linearly dependent over $\UU^p$ indeed.

After possibly renumbering, we may take a smallest $m\geqslant 2$ such that $a_1,\ldots,a_m$ are linearly dependent over $\UU^p$ and each proper subset of $\{a_1,\ldots,a_m\}$ is linearly independent over $\UU^p$. Let us take again $\lambda_1,\ldots,\lambda_n$ in $\UU$  such that
\[\lambda_1^p a_1+\cdots+\lambda_m^p a_m=0.\]
Applying $\partial$, we obtain
\[0=\partial(0)=\partial \left( \sum_{i=1}^m \lambda_i^p a_i\right) =  \sum_{i=1}^m \lambda_i^p a_i^{p^{k_i}} a_i.\]


By minimality of $m$, we can conclude that
$$a_1^{p^{k_1}}=a_2^{p^{k_2}}=\cdots= a_m^{p^{k_m}}$$
Since $m\geqslant 2$ and $a_1\neq a_2$, we must have that $k_1\neq k_2$. We may assume $k_1<k_2$. This implies that $a_1=a_2^{p^{k_2-k_1}}\in \UU^p$, a contradiction.
\end{claimproofstar}

In order to show now that the elements $a_1,\ldots, a_n$ are $p$-independent, assume otherwise and choose a non-trivial $p$-dependence $P(\bar X)$ of least possible length (i.e., smallest number of $p$-monomials appearing nontrivially in $P$), so \[ 0=P(\bar a)=\sum_{\alpha} \lambda_{\alpha}^p w_{\alpha}(\bar a)\] for some elements $\lambda_{\alpha}$ in $\UU$ and $p$-monomials \[ w_{\alpha}(\bar X)= \prod_{i=1}^{n} X_i^{\alpha_i}.\]

 Now,
  \begin{align*} 0=\partial(P(\bar a))= {} & \sum_{\alpha} \lambda_{\alpha}^p \partial(w_{\alpha}(\bar a))\stackrel{\ref{R:p-mon_diff}}{=}  \sum_{\alpha} \lambda_{\alpha}^p \big(\sum_{i=1}^n \alpha_i a_i^{p^{k_i}} \big) w_{\alpha}(\bar a)\end{align*}



Our minimal choice of the length of $P$ implies that for any $\alpha$ and $\beta$, with $\lambda_\alpha$ and $\lambda_\beta$ appearing in $P$, we have
 $$\sum_{i=1}^n\alpha_i a_i^{p^{k_i}}=\sum_{i=1}^n\beta_i a_i^{p^{k_i}}$$
 Taking $\alpha$ and $\beta$ such that $\alpha_1\neq \beta_1$, we can re-write the above as
 $$\gamma_1 a_1^{p^{k_1}}+\gamma_2 a_2^{p^{k_2}}+ \cdots + \gamma_n a_n^{p^{k_n}}=0$$
 with $\gamma_1\neq 0$ and $\gamma_i\in \mathbb F_p$. We may assume that $k_1= k_i$ for all $1\leqslant i\leqslant i_0$ and $k_1<k_i$ for all $i_0+1\leqslant i\leqslant n$. Hence, the above yields
 $$\gamma_1 a_1+\cdots+\gamma_{i_0}a_{i_0}\in \UU^p$$
 Since $\gamma_i=\gamma_i^p$ and $\gamma_1\neq 0$, this contradicts the Claim.
 \end{proof}

Proposition \ref{P:trivial_new} together with Theorem \ref{T:main} and Remark \ref{R:remark_indstar} (c) yield now the following result.
\begin{cor}\label{C:trival_indep}
Given pairwise distinct nonzero solutions $a_1,\ldots,a_n$ of the equation
$\partial(X)=X^{p^k+1}$ for $k\geq 1$, the elements $a_1,\ldots,a_n$ are independent in the theory $\DCF_p$; that is, \[  a_i\inddcf a_1,\ldots, a_{i-1} \ \text{ for all } 2\leqslant i\leqslant n.\]
\end{cor}

To analyze the induced structure on the definable set $\{x\in \UU: \partial(x)=x^{p^k +1}\}$, we need the following fact which may be folklore.
\begin{prop}\label{strdis}
Suppose that $V$ is a set definable in a model of a stable theory which isolates a complete stationary type and such that forking on $V$ coincides with equality. Then, $V$ is \emph{strictly disintegrated}. Namely, $V$ has the induced structure of a pure set. In particular, $V$ is strongly minimal.
\end{prop}
\begin{proof}
Since (by stability) $V$ is stably embedded, it is enough to show that for any $a_1,b_1,\ldots,a_n,b_n\in V$, we have
$$|\{a_1,\ldots,a_n\}|=n=|\{b_1,\ldots,b_n\}|\quad \Longrightarrow\quad \tp(a_1,\ldots,a_n)=\tp(b_1,\ldots,b_n).$$
However, the cardinality assumption and the assumption on forking on $V$ implies that both $(a_1,\ldots,a_n)$ and $(b_1,\ldots,b_n)$ are Morley sequences in the stationary complete type isolated by $V$, so the result follows.
\end{proof}
\begin{theorem}\label{T:type_trivial}
The equation $\partial(T)=T^{p^k+1}$ for $k\geq 1$ defines a strongly minimal set which is strictly disintegrated; in particular, it is geometrically trivial.

Moreover, if $k<l$, then the set of solutions of $\partial(T)=T^{p^k+1}$ is orthogonal to the set of solutions of $\partial(T)=T^{p^l+1}$.
\end{theorem}
\begin{proof}
The strict disintegration follows immediately from Proposition \ref{strdis} together with Corollary \ref{C:trival_indep} and Remark \ref{R:isol} (a).

By \cite[Chapter 2, Lemma 4.3.1(iii)]{PillayBook}, two stationary types $p$ and $q$ defined over the same base set are orthogonal if and only if all for all $n$ and $m$, the types $p^{\otimes n}$ and $q^{\otimes m}$ are weakly orthogonal, where $p^{\otimes n}$ is the stationary type given by a Morley sequence of $p$ of length $n$. Hence, the moreover part follows from Proposition \ref{P:trivial_new}.
\end{proof}

\begin{remark}
We conclude by making a few notes on how our examples of geometrically trivial strongly minimal sets compared to the known examples in the cases of the theories DCF$_0$ and SCF$_{p,\infty}$.
\begin{enumerate}[(a)]
\item One may compare the transcendence statement from Proposition \ref{P:trivial_new} with the corresponding results in the characteristic 0 case about strongly minimal and geometrically trivial definable sets (see e.g. \cite{cfn1, remi, np14}) which also give ``Ax-Schanuel like'' consequences. We would also like to point out that our proof is quite different than the corresponding proofs of strong minimality of differential equations in characteristic 0: we show $p$-independence of distinct solutions first,  which then gives their algebraic independence and strict disintegration of the set of solutions (directly implying strong minimality and triviality) almost immediately.

\item It is rather straightforward now to see what happens in the presence of parameters. By Theorem \ref{T:type_trivial} and the descripion of the model theoretic algebraic closure in DCF$_p$ (see Corollary \ref{C:diffperf} (b)), a solution $a$ of $\partial(T)=T^{p^k+1}$ is generic over a differentially perfect differential subfield of $\UU$ if and only if it is transcendental over this subfield.

\item An infinite set which is $\SCF_{p,\infty}$-definable cannot have finite U-rank or even be superstable: see part (4) of the remark after Proposition 2.11 in \cite{fD98} which is stated in the case of finite imperfection degree, but the argument works in the case of $\SCF_{p,\infty}$ as well. Therefore, our strictly disintegrated example is ``very different'' from $\SCF_{p,\infty}$-definable sets.

    However, it can be easily shown (using Leibniz's method) that if $n$ is not congruent to $1$ modulo $p$, then the set
  \[\{x\in \UU\ |\ \partial(x)=x^n\}\]
  is $\SCF_{p,\infty}$-definable over one additional parameter.

\item We could not find in the literature any examples of trivial types of U-rank one in the theory $\SCF_{p,\infty}$, so, in this sense, our strictly disintegrated example is also new.
\end{enumerate}
\end{remark}

\section*{Appendix: A general criterion for stability}\label{secgenst}

\numberwithin{equation}{section}
 \setcounter{section}{5}

We assume we have two countable languages $\LL\subseteq \LL'$, a complete $\LL$-theory $T$, and a complete $\LL'$-theory $T'$ such that $T\subseteq T'$. Let us fix a monster model $\UU\models T'$, so it is also a monster model of $T$. All $\LL$-structures and $\LL'$-structures considered are assumed to be substructures of $\UU$. For any $A\subset \UU$, let $\langle A\rangle_{\LL'}$ denote the $\LL'$-substructure of $\UU$ generated by $A$.

The proof of the following result comes from the proof of \cite[Theorem 2.4]{pK05} and is inspired by Shelah's proof of
\cite[Theorem 9]{sS73}.

\begin{theorem}\label{genstab}
The conditions listed below imply that the theory $T'$ is stable.
\begin{enumerate}
  \item The theory $T$ is stable (and let us denote the corresponding forking-independence relation by $\ind$).

  \item The theory $T'$ has quantifier elimination.

  \item $($\emph{the $\ind$-coproduct condition}$)$  Suppose that
  \begin{itemize}
    \item $M,M',N\models T'$;

    \item $M\cap N=M'\cap N=:N_0$;

    \item $M\ind_{N_0}N$ and $M'\ind_{N_0}N$;

    \item there is an $\LL'$-isomorphism  over $N_0$
    $$f:M\xrightarrow{\cong} M'.$$

  \end{itemize}
 Then, $f$ extends to an $\LL'$-isomorphism over $N$
 $$\tilde{f}:\langle M\cup N\rangle_{\LL'}\xrightarrow{\cong} \langle M'\cup N\rangle_{\LL'}.$$
 \end{enumerate}
\end{theorem}
\begin{proof}
Let us fix $N\models T'$.
\begin{claim}\label{Cl:asinlocal}
For any $a\in \UU$, there is a countable $\LL'$-substructure $M_a\subset \UU$ such that $a\in M_a$ and
$$M_a\ind_{M_a\cap N}N.$$
\end{claim}
\begin{claimproof}
The proof is exactly the same as the proof of the construction of $K'$ and $\ell$ in the proof of Claim \ref{Cl:locchar} (local character of $\ind^*$) from the proof of Theorem~\ref{T:main}.
\end{claimproof}
For any $a\in \UU$, we fix a countable $\LL'$-structure $M_a$ as in the statement of Claim~\ref{Cl:asinlocal}. We define the following equivalence relation on $\UU$. For any $a,b\in \UU$, we set $a\sim b$ if and only if
$$M_a\cap N=M_b\cap N=:N_{ab},$$
and there is an $\LL'$ isomorphism
$$f:M_a\xrightarrow{\cong}N_a$$
 over $N_{ab}$ such that $f(a)=b$.
\begin{claim}\label{Cl:coarser}
For any $a,b\in \UU$, if $a\sim b$ then $\tp_{\LL'}^{\UU}(a/N)=\tp_{\LL'}^{\UU}(b/N)$.\end{claim}
\begin{claimproof}
Assume that $a\sim b$ and let $N_{ab}$ and $f$ be as above. We are now exactly in the situation from the ``$\ind$-coproduct condition'' assumption, so $f$ extends to an $\LL'$-isomorphism
 $$\tilde{f}:\langle M_a\cup N\rangle_{\LL'}\xrightarrow{\cong} \langle M_b\cup N\rangle_{\LL'}$$
 over $N$. Since $\tilde{f}(a)=b$, we obtain
$$\tp_{\LL'}^{M_aN}(a/N)=\tp_{\LL'}^{M_bN}(b/N).$$
Since $T'$ has quantifier elimination, we get $\tp_{\LL'}^{\UU}(a/N)=\tp_{\LL'}^{\UU}(b/N)$.
\end{claimproof}
It is clear that the relation $\sim$ has at most $|N|^{\aleph_0}$ equivalence classes. By Claim~2, the relation $\sim$ is coarser than the equality of $\LL'$-types calculated in $\UU$, so $T'$ is stable.
\end{proof}

\begin{remark} \label{R:criterion_consequences}
\begin{enumerate}
\item One can easily check that Theorem \ref{genstab} directly implies the stability of DCF$_{p,m}$ (taking $T=$ SCF$_{p,\infty}$).
  \item We expect that conditions $(1)$--$(3)$ from the statement of Theorem \ref{genstab} should be easy to check in suitable theories of fields with operators,
  where $T=\ACF$ or $T=\SCF_{p,\infty}$, since linear disjointness implies that the compositum comes from the tensor product (being the appropriate coproduct in this context).

  \item In the unstable theory ACFA, one cannot put ``quantifier elimination and the $\ind$-coproduct condition'' together.

\item As Theorem \ref{T:main} indicates, it is hard to expect that one can improve the conclusion of Theorem \ref{genstab} to include a \emph{description} of the forking-independence in the theory $T'$. It is still possible that one gets a description over models similar to the one in Fact \ref{F:Shelah}, but we do not know how to prove (or disprove) it at the moment.

 \item Our proof of Theorem \ref{genstab} is inspired by Shelah's proof of stability of DCF$_p$ (see \cite[Theorem 9]{sS73}). However, Shelah's proof is more difficult, since he uses only stability of the theory ACF$_p$ rather than stability of the theory SCF$_{p,\infty}$. That is why Shelah had to prove and use an additional result from differential algebra (cf. \cite[Lemma 7.6]{MPZ20}) to conclude his proof. Historically, the stability of SCF$_{p,\infty}$ was an immediate consequence of stability of DCF$_p$ as observed by Macintyre. A separate proof of stability of SCF$_{p,\infty}$ (still informed by the proof of \cite[Theorem 9]{sS73}) was found later by Macintyre, Shelah, and Wood (see \cite[Theorem 3]{cW79}).

\end{enumerate}
 \end{remark}

\end{document}